\documentclass[12pt,a4paper]{amsart}
\usepackage{amsfonts}
\usepackage{amsthm}
\usepackage{amsmath}
\usepackage{amscd}
\usepackage{pinlabel}

\usepackage[utf8]{inputenc}
\usepackage{t1enc}
\usepackage[mathscr]{eucal}
\usepackage{indentfirst}
\usepackage{graphicx}
\usepackage{graphics}
\usepackage{pict2e}
\usepackage{epic}
\numberwithin{equation}{section}
\usepackage[margin=2.9cm]{geometry}
\usepackage{epstopdf} 
\usepackage{hyperref}

\hypersetup{
	colorlinks = true,
	urlcolor   = blue,
	citecolor  = blue,
} 
\usepackage{tikz-cd}
\usepackage{verbatim}
\usepackage{float}
\usepackage{makecell}
\begin{document}
	\newtheorem{de}{Definition}
	\newtheorem{ex}[de]{\emph{Example}}
	\newtheorem{thm}[de]{Theorem}
	\newtheorem{lemma}[de]{Lemma}
	\newtheorem{cor}[de]{Corollary}
	\newtheorem{con}[de]{Conjecture}
	\newtheorem{prop}[de]{Proposition}
	\newcommand{\sla}[2]{\ {\ }_{#1}  /^{{#2}}}
    \title{Efficient geodesics in the curve complex and their dot graphs}
	
	\author{Hong Chang}
	
	\address{
		\noindent
		Hong Chang, hchang24@buffalo.edu, Department of Mathematics, University at Buffalo--SUNY}

	
	\begin{abstract}
	For the complex of curves of a closed orientable surface of genus $g$, $\mathcal{C}(S_{g>1})$, the notion of {\em efficient geodesics} in was introduced in \cite{BMM}.  There it was established that there always exists (finitely many) efficient geodesics between any two vertices, $ v_{\alpha} , v_{\beta} \in \mathcal{C}(S_g)$, representing homotopy classes of simple closed curves, $\alpha , \beta \subset S_g$. The main tool for used in establishing the existence of efficient geodesic was a {\em dot graph}, a booking scheme for recording the intersection pattern of a {\em reference arc}, $\gamma \subset S_g$, with the simple closed curves associated with the vertices of geodesic path in the zero skeleton, $\mathcal{C}^0(S_g)$.  In particular, for an efficient geodesic between $v_\alpha $ and $v_\beta$ of length $d \geq 3$, it was shown that any curve corresponding to the vertex that is distance one from $v_\alpha$ intersects any $\gamma$ at most $d -2$ times.  In this note we make a more expansive study of the characterizing ``shape'' of the dot graphs over the entire set of vertices in an efficient geodesic edge-path.  The key take away of this study is that the shape of a dot graph for any efficient geodesic is contained within a {\em spindle shape} region.  Since the Nielson-Thurston coordinates of any curve on $S_g$ are directly derived from its intersection number with finitely many reference arcs, spindle shaped dot graphs control the coordinate behavior of curves associated with the vertices of an efficient geodesic.
	\end{abstract}
	\maketitle
	\section{Introduction}
	
	The complex of curves, $\mathcal{C}(S)$ for a closed surface $S$, is the simplicial flag complex whose vertices correspond to isotopy classes of essential simple closed curves in $S$ and whose edges connect vertices with disjoint representatives.  When $S$ is a closed oriented surface of genus $g \geq 2$, the $1$-skeleton, $\mathcal{C}^1(S)$, will be path-connect. Thus, the $0$-skeleton, $\mathcal{C}^0 (S)$, can be endowed with a metric by defining the distance between two vertices to be the minimal number of edges in any edge path in $\mathcal{C}^1(S)$ between the two vertices. With this metric in hand, the nature of the coarse geometry of $\mathcal{C}(S)$ has been extensively studied and has deep applications to 3-manifolds, mapping class groups, and Teichm\"uller space.  (See for example \cite{M}.) The seminal result, due to Masur and Minsky in 1996, states that $\mathcal{C}^0 (S)$ is $\delta$-hyperbolic \cite{MM}.  Since Masur-Minsky's initial hyperbolicity result, it has been shown that $\delta$ can be chosen independently of of the genus of $S$. (See \cite{A, B, CRS, HPW, PS}.)

    Computing actual distances between vertices in the curve complex has proved to be a subtle task.  Initial progress for giving an effective algorithm to compute distance between two vertices of $\mathcal{C}(S)$ was made by Leasure \cite{Lea}, followed by Shackleton \cite{S}, Webb \cite{We}, and Watanabe \cite{Wa}.  More recently, using the notion of {\em tight geodesics} coming from the work of Masur-Minsky \cite{MM}, Bell and Webb \cite{BW} have given an algorithm the described an algorithm that computes distance in polynomial time.  
    
    The efforts of this note is based upon Birman, Maralgit and Menasco's innovative work on {\em efficient geodesics} in $\mathcal{C}(S)$ \cite{BMM}.  Notably in \cite{MICC} a partial implementation of an algorithm based on efficient geodesics is available. Further work on efficient geodesic can be found in \cite{BMW} and \cite{Jin}.
    \smallskip
    \smallskip
    
    \noindent {\em Efficient geodesics--} Let $ v_0, \cdots , v_n$ be a geodesic path in $\mathcal{C}^0(S)$ with $n \leq 3$, that is $d(v_i , v_j) = |i - j|$ where $d$ is the edge distance function on $\mathcal{C}^0(S)$. Let $\{\alpha_0 , \alpha_1,  \alpha_n\} \subset S$ be the curves representing the isotopy classes of $v_0, v_1, v_n$, respectively, that are pairwise in minimal position (this configuration is unique up to isotopy of $S$).
    A reference arc for the triple $\alpha_0, \alpha_1, \alpha_n$, is an arc $\gamma \subset S \setminus (\alpha_0 \cup \alpha_n)$ that is in minimal position with $\alpha_1$.  Then the oriented geodesic $v_0, \cdots , v_n$ is {\em initially efficient} if $|\alpha_1 \cap \gamma | \leq n - 1$ for all choices of reference arcs $\gamma$. Next, $ v = v_0 , \cdots , v_n = w $ is {\em efficient} if the oriented geodesic $v_k , \cdots , v_n$ is initially efficient for each $0 \leq k \leq n - 3$ and the oriented geodesic $v_n,v_{n-1},v_{n-2},v_{n-3}$ is also initially efficient.  Thus, we have the triple $v_k,v_{k+1}, v_n$ accounting for $n-k-1$ points of intersection of (a representative of) $v_{k+1}$ with any such reference arc $\gamma$.  As observed in \cite{BMM}, while there are infinitely many reference arcs, there is a natural reduction to a finite collect of reference arcs which need to be checked in the definition of initial efficiency.
    \smallskip
    \smallskip
    
    \noindent {\em The dot graph and its surgeries--}
    The main tools in \cite{BMM} for establishing the existence of efficient geodesics are surgeries on the {\em dot graph} of a geodesic path.  Specially, given a vertex path $v_0, \cdots , v_n \subset \mathcal{C}(S)$ with representative curves $\alpha_0, \cdots , \alpha_n \subset S$ and an oriented reference arc $\gamma$, we obtain a corresponding graph by placing a point at the coordinate, $(k,i) \in \mathbb{R}^2$, if the $k^{\rm th}$ intersection along $\gamma$ is $\alpha_i$.  The set of dots can be arranged into a collections of ascending line-segments of dots, i.e. a sawtooth pattern.  Once so arranged, it is readily discernible when a given sawtooth pattern admit one three possible surgeries---box and hexgon type $1$ and $2$ surgeries---that reduces the intersection of the $\alpha_i {\rm 's}$ with $\gamma$.  Our contribution to this machinery is an additional surgery---the double box surgery.
    \smallskip
    \smallskip
    
    \noindent {\em Dot graphs, spindles and main result--}
    It is possible that although a geodesic is efficient, its dot graph may still admit the application of one of our four surgeries.  As such, we consider {\em irreducible} geodesics, those for which their dot graph does not admitted a box, hexagon (type 1 or 2) or double-box surgery.
    
    \emph{Spindles} are a type of dot graphs that has the ascending line-segments grouped together to form an ``upper triangle'' portion and a ``lower triangle'' portion plus a maximal length line-segment between these two triangles.  See Fig. \ref{figure: spindle} for an example of this grouping.  The salient feature of a spindle is that it does not admit one of our four surgeries.  Please see \S \ref{subsection: sawtooth} and \ref{section: finding surgeries} for precise definitions of surgeries and spindle.
    
    \begin{figure}[h]
\labellist

\pinlabel ${\rm lower \ triangle}$ at 465 120

\pinlabel ${\rm upper \ triangle}$ at 45 243

\pinlabel ${\rm maximal \ \ segment}$ at 312 243

\endlabellist
    \centering
    \includegraphics[width=.8\linewidth]{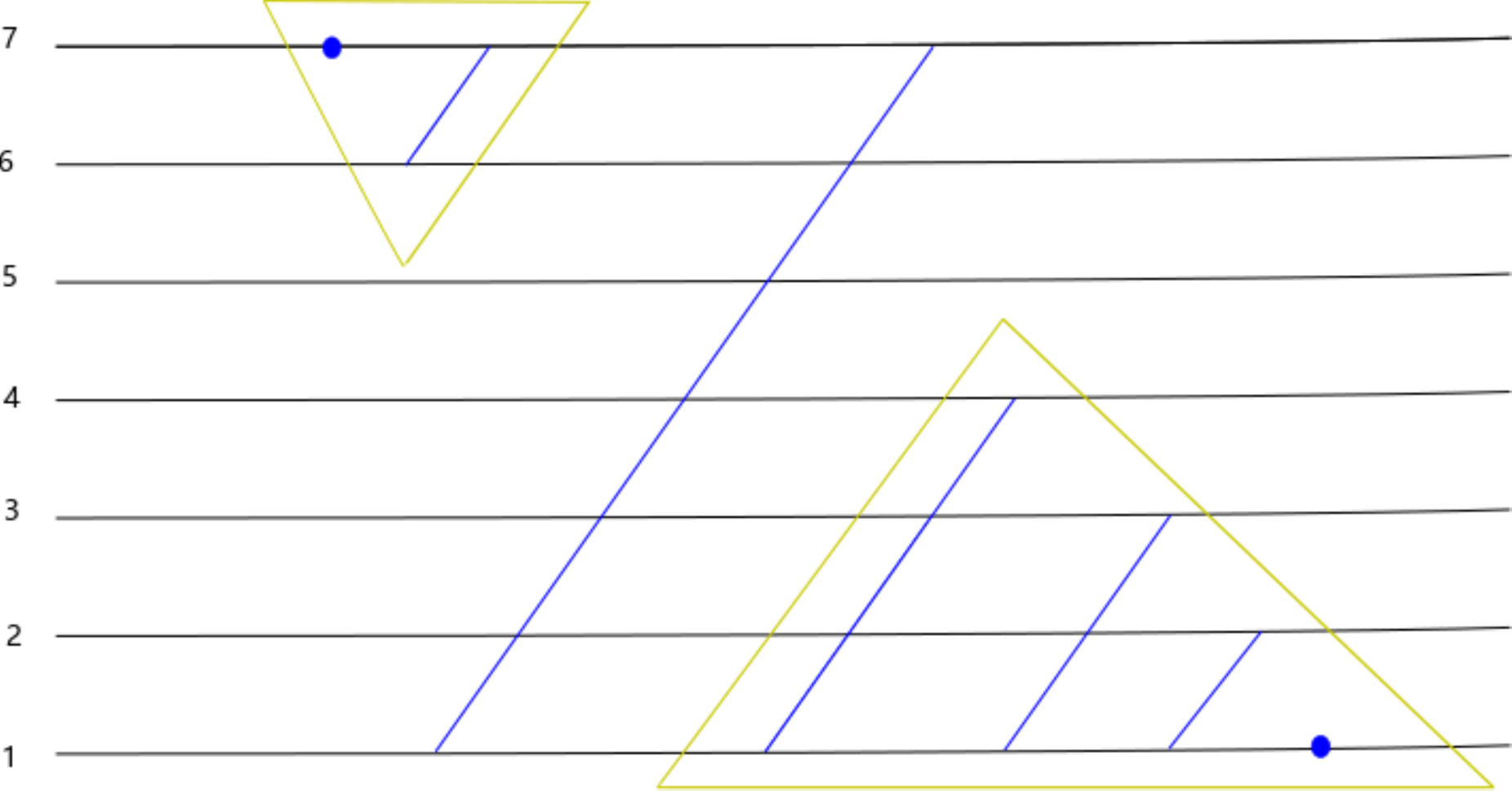}
		\caption{A spindle that is made up of three parts: the upper triangle, the maximum segment and the lower triangle}
		\label{figure: spindle}
	\end{figure}
    \begin{thm}
		The dot graph of any irreducible geodesic is a sub-graph of a spindle.
	\end{thm}

    \subsection{Outline of paper}
    In \S~\ref{section: prliminary}, we introduce and expand on the machinery of \cite{BMM}, including the dot graph, the sawtooth form,  and the reducibility criterion. Then we give a new surgery that will be used to prove our main theorem. In \S~\ref{section: proof of theorem}, we prove our main theorem. In \S~\ref{section: finding surgeries}, we describe an algorithm to find possible surgeries for dot reducible dot graphs. In \S~\ref{section: tree}, we introduce reference tree and discuss its reducibility.
    
    \section*{Acknowledgements}
    I would like to thank my advisor William Menasco for showing me the concept of efficient geodesics and suggesting the $\sla{a}{b}$ notation for expressing line-segments in the dot graphs.  He helped me a lot on writing this paper, especially the introduction part.
    
    \section{Preliminary material}
    \label{section: prliminary}
    Let $\alpha_0 , \alpha_n \subset S$ be two simple closed curves (s.c.c.) representing the isotopy classes, $v_0, v_n$, respectively, with distance, $d(v_0 , v_n) \geq 3$.  Then $(\alpha_0, \alpha_n)$ is a {\em filling} pair.  We assume the geometry intersection number, $i(\alpha_0 , \alpha_n)$ is minimal.   Below is the setup in Section 3 of \cite{BMM}, the readers who are familiar with this reference may skip ahead to \S \ref{subsection: sawtooth}.
    
    \subsection{Complexity and a reducibility criterion}
    Let $v_0,v_1,...,v_n$ be a path in $\mathcal{C}^0(S)$ and $\alpha_0,\alpha_1,...,\alpha_n$ be the associated s.c.c.'s. Borrowing from Leasure \cite{Lea}, we have the concept of a \emph{reference arc, $\gamma$ of $\alpha_0,\alpha_1,...,\alpha_n$}.  Specifically, $\gamma$ satisfies the following:
	
	(1) $\gamma$ has its interior disjoint from $\alpha_0\cup\alpha_n$
	
	(2) $\gamma$ has endpoints disjoint from $\alpha_1,...,\alpha_{n-1}$,
	
	(3) all triple intersections $\alpha_i\cap\alpha_j\cap\gamma$ are trivial for $i\not=j$, and
	
	(4) $\gamma$ is in minimal position with each of $\alpha_0,\alpha_1,...,\alpha_n$.
	
	For $n \geq 3$, when $v_0, \cdots , v_n$ is a geodesic, we say the geodesic is {\em initially efficient} if $|\alpha_1 \cap \gamma | \leq n - 1$ for all choices of reference arcs $\gamma$. Next, $ v = v_0 , \cdots , v_n = w $ is {\em efficient} if the oriented geodesic $v_k , \cdots , v_n$ is initially efficient for each $0 \leq k \leq n - 3$ and the oriented geodesic $v_n,v_{n-1},v_{n-2},v_{n-3}$ is also initially efficient.  Thus, we have the triple $v_k,v_{k+1}, v_n$ accounting for $n-k-1$ points of intersection of (a representative of) $v_{k+1}$ with any such reference arc $\gamma$.
	
	We define \emph{complexity} of the path $v_0, \cdots , v_n$ as $$\sum_{k=1}^{n-1}(i(v_0,v_k)+i(v_k,v_n))$$where $i(u,v)$ is the minimal intersection number over all representatives of $u$ and $v$ that are in minimal position. \cite{BMM} showed geodesics with minimal complexity are naturally efficient.
	
	There is relation between complexity and the intersections between the $\alpha_1,...,\alpha_{n-1}$ and the reference arc: suppose the curve $\alpha_i$ intersects $\gamma$ which lies inside a disk bounded by arcs of $\alpha_0\cup\alpha_n$, then $\alpha_i$ must also intersects twice with the boundary of the disk, contributing the complexity.
	
	The lower bound for the complexity of geodesics of distance $d$ is easy to find: From \cite{Hem} we know that for the distance and intersections of two curves $\alpha$ and $\beta$ we have $$d(\alpha,\beta)\le 2log_2(i(\alpha,\beta))+2,$$ so  the minimal intersecting number of two curves with distance $d$ is $2^{\frac{d-2}{2}}$ so the minimal complexity is $$\sum_{k=1}^{n-1}(i(v_0,v_k)+i(v_k,v_n))\ge 2\sum_{k=2}^{n-1}2^{\frac{k-2}{2}}$$
	
	When $n$ is small, the lower bound is obtainable. For example, when $n=4$, Fig. \ref{Figure: Distance 4} reaches the theoretical minimal $2\cdot(1+4)=10$. Fig. \ref{Figure: Distance 4} is the distance $4$ filling pair given in \cite{MICC}.
	
	\begin{figure}[h]
		\scalebox{.5}{\includegraphics[angle=0,origin=c]{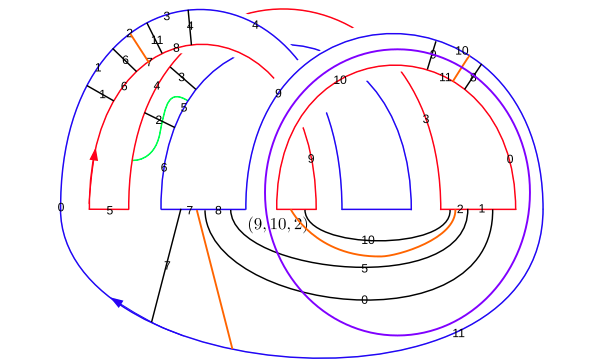}}
		\vspace{0cm}
		\caption{Example of a geodesic with distance 4. The blue is glued with the red, forming $\alpha_0$ and the black is $\alpha_4$ which fills and uniquely determined a surface $S_2$ of genus 2. The violet is $\alpha_1$, the green is $\alpha_2$ and the orange is $\alpha_3$.}
		\label{Figure: Distance 4}
	\end{figure}
	
	Now we are looking for an upper bound for the complexity. In general, the complexity can be arbitrary large for a geodesic connecting two curves $\alpha_0$ and $\alpha_n$. In order to give a bound, we only consider geodesics with no possible surgeries to reduce the complexity.
    \subsection{Sawtooth form and the dot graph}
    \label{subsection: sawtooth}

	A \emph{dot graph} is a graphical representation of the sequence of vertices $v_i$ seen along a reference arc; there is a dot
	at the point $(k,i)$ in the plane if the $k$th vertex along the arc is $v_i$. A dot graph (or its sequence) is called in \emph{sawtooth form} if the sequence $(j_1,...,j_k)$ satisfy$$j_i<j_{i+1}\Rightarrow j_i+1=j_{i+1}$$Lemma 3.3 in \cite{BMM} showed any dot graph can be modified into sawtooth form using commutations.  To enhance associated sequences we connect the ascending dots with a line-segment.  See Fig. \ref{figure: sawtooth graph} for an example.
	
	\begin{figure}[h]
		\scalebox{1}{\includegraphics[angle=0,origin=c]{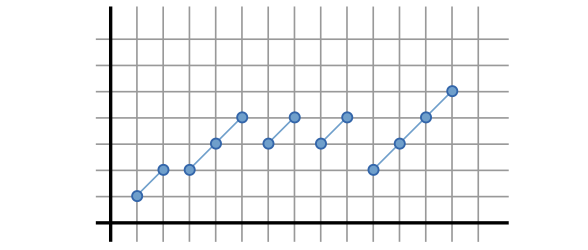}}
		\vspace{0cm}
		\caption{Example of sawtooth form dot graph in \cite{BMM}.}
		\label{figure: sawtooth graph}
	\end{figure}
	
	 In a saw-tooth dot graph, if an ascending line-segment starts at $p$ and end at $q$, we denote the segment as $\sla{p}{q}$ ($p$ can be equal to $q$ and it would become single dot); if the graph (from left to right) is made of ascending line-segments $\sla{p_1}{q_1}, \cdots, \sla{p_k}{q_k}$, then we denote the dot graph $$\bigg[\sla{p_1}{q_1},...,\sla{p_k}{q_k}\bigg]$$.
	 
	 For example, with the new notation, the sawtooth in Fig. \ref{figure: sawtooth graph} will be denoted as $\bigg[\sla{2}{3},\sla{3}{5},\sla{4}{5},\sla{4}{5},\sla{3}{6}\bigg]$.

\begin{de}
\label{def: spindle}
      \emph{Let $v_0, \cdots , v_n$ be a path in $S$ and along the reference arc $\gamma$ we see a dot graph. A}  lower complete triangle dot graph \emph{is defined as} $$\bigg[\sla{1}{n-1},\sla{1}{n-2},...,\sla{1}{1}\bigg]$$ \emph{and a} upper complete triangle \emph{is defined as} $$\bigg[\sla{n-1}{n-1},\sla{n-2}{n-1},...,\sla{1}{n-1}\bigg].$$
      \emph{Let $0\le k\le n-2$ be an integer, we call a dot graph $$\bigg[\sla{n-1}{n-1},...,\sla{k+2}{n-1},\sla{1}{n-1},\sla{1}{k},\sla{1}{k-1},...,\sla{1}{1}\bigg]$$ a} spindle at $k$. \emph{Notice lower complete triangle and upper complete triangle are degenerated spindle with $k=n-2$ and $k=0$.}
\end{de}

    \subsection{Surgeries}
    All surgeries occur in a neighborhood of $\gamma$.  A \emph{surgery} of $\alpha$ with respect of $\gamma$ is an operation that, first, cuts $\alpha$ into arcs $\alpha \setminus \alpha\cap\gamma$ and, second, reconnects some of them (and deletes the others) with arc $\delta$ in the neighborhood of $\gamma$ such that the result is a curve. A surgery of a path $\alpha_0,\alpha_1,...,\alpha_n$ with respect of $\gamma$ is a set of surgeries of some of the curves such that the result is still a path (i.e. every two adjacent curves is distance 1 in $C(S_g)$). A \emph{dot reducible} dot graph is a dot graph that we can make a surgery to reduce the number of its entries.
    
    In \cite{BMM}, if we let a neighborhood of $\gamma$ in the plane so that $\gamma$ is a horizontal arc oriented to
the right. We say that $\alpha$ is obtained from $\alpha$ by $++, +-, -+$, or $--$ surgery along $\gamma$ ;
the first symbol is $+$ or $-$ depending on whether the first endpoint of $\delta$ (as measured
by the orientation of $\gamma$ ) lies above $\gamma$ or below, and similarly for the second symbol. The reader is referred to Fig.10 in that paper for a graph of it. box and hexagon surgeries for dot graphs are also introduced in \cite{BMM}, here are they:  Fig. \ref{figure: box and hexagon}. These figures are from \cite{BMM} and the author would like to thank them for providing it. With the new notation, a \emph{box} can be written as $\bigg[\sla{p_1}{q_1},\sla{p_2}{q_2}\bigg]$ where $p_1\le p_2$ and $q_1\le q_2$. A \emph{hexagon} of type 1 can be written as $\bigg[\sla{p_1}{q_1},\sla{p_2}{q_2},\sla{p_3}{q_3}\bigg]$ where $p_1\le p_2$, $q_1\le q_3$ and $p_3\le q_2$. A \emph{hexagon} of type 2 can be written as $\bigg[\sla{p_1}{q_1},\sla{p_2}{q_2},\sla{p_3}{q_3}\bigg]$ where $p_1\le p_3$, $q_2\le q_3$ and $p_2\le q_1$. The reader is referred to \cite{BMM} for their surgeries.
    
    \begin{figure}[h]
		\scalebox{1}{\includegraphics[angle=0,origin=c]{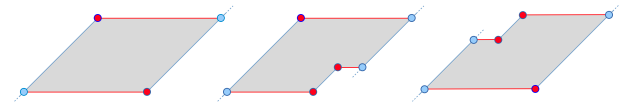}}
		\vspace{0cm}
		\caption{A box, a hexagon of type 1, and a hexagon of type 2; the red dots are required to be
endpoints of ascending segments, while the blue dots may or may not be endpoints}
		\label{figure: box and hexagon}
	\end{figure}
	 
	In this paper, we will also introduce two types of new surgeries in addition to the two surgeries above, the double box surgery and the double hexagon surgery.
	
	\quad A \emph{double box} in $G(\sigma)$ is $\sigma$-octagon $P$ with the following four properties:
	\begin{enumerate}
		\item There is an ascending segment which is neither leftmost nor rightmost that contains the highest and lowest point of $G(\sigma)$,
		\item the leftmost ascending segment contains the highest point of some ascending segment of $G(\sigma)$,
		\item the rightmost ascending segment contains the lowest point of some ascending segment of $G(\sigma)$,
		\item the lowest point of leftmost ascending segment is at most one grid higher than the highest point of rightmost ascending segment.
	\end{enumerate}
	With the new notations, double box can be written as $\bigg[\sla{p_1}{q_1},\sla{p_2}{q_2},\sla{p_3}{q_3}\bigg]$, where $q_1\le q_2$, $p_2\le p_3$ and $p_1\le q_3+1$.
	\begin{figure}[h]
		\scalebox{.5}{\includegraphics[angle=0,origin=c]{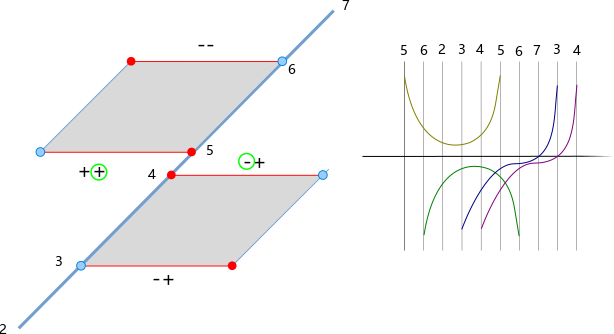}}
		\vspace{0cm}
		\caption{A double box and its surgery. The new 6 does intersect with 3 and 4, but it does not matter.   }
		\label{double box}
	\end{figure}
	
	For double box, we can make a surgery just like a regular box. We make surgery from lowest to highest following the directed graph Fig.12 in \cite{BMM}. Notice when the direction changes we need to make sure the left of the bottom one is different from the right of the top one (marked in the figure), unlike in Fig.12 of \cite{BMM}. Notice in this surgery the box can be pierced.
   
    \begin{figure}[h]
		\scalebox{.5}{\includegraphics[angle=0,origin=c]{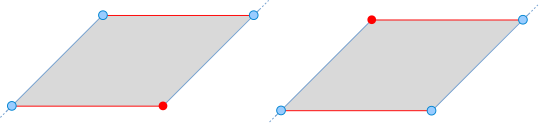}}
		\vspace{0cm}
		\caption{Left and right boxes when the circled dots are required to be endpoints while the crossed dots are not. When crossed dots are endpoints, they become real boxes.}
		\label{Left and right boxes}
	\end{figure}
	\quad A \emph{left half box} in a $G(\sigma)$ is a $\sigma$-$G(\sigma)$-quadrilateral with only the condition "the leftmost ascending segment contains the highest point of some ascending segment of $G(\sigma)$" while a \emph{right half box} only satisfy "the rightmost ascending segment contains the lowest point of some ascending segment of $G(\sigma)$". With the new notations, left box can be written as $\bigg[\sla{p_1}{q_1},\sla{p_2}{q_2}\bigg]$, where $p_1\le p_2$ and right box can be written as $\bigg[\sla{p_1}{q_1},\sla{p_2}{q_2}\bigg]$, where $q_1\le q_2$. We can't perform surgeries with the half boxes however we have:
	
	\quad A \emph{double hexagon} is a combination of a left half box in the left side and a right half box in the right side.
	\begin{figure}[h]
		\scalebox{.45}{\includegraphics[angle=0,origin=c]{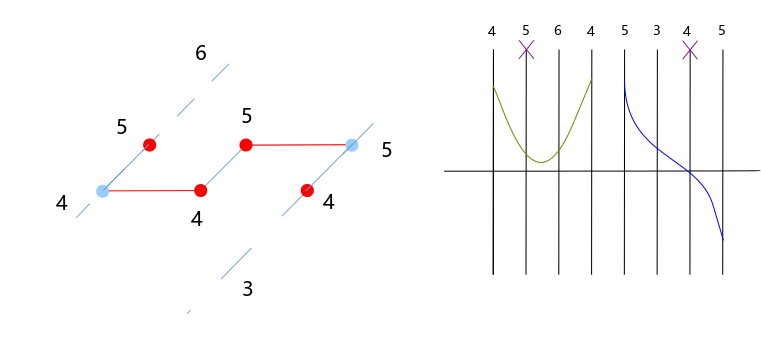}}
		\vspace{0cm}
		\caption{A double hexagon and its surgery, crossed dots are eliminated. Notice the left and right surgeries are not "connected" so they can preform at the same time}
		\label{figure: double hexagon}
	\end{figure}
	For surgeries of double hexagon, we need to do two surgeries at the same time: first we do the surgery in the left to eliminate the lowest point in the right, then we do the same in the right to eliminate the highest point in the left. The two surgeries not "connected" (see the discussion in \S \ref{section: finding surgeries} so they can preform at the same time.
    
	\section{The proof of the theorem}
	\label{section: proof of theorem}

	Dot reducibility is not the same as reducibility mentioned in \cite{BMM}. However we have the following lemma.
	
	\begin{lemma}
	\label{lemma: dot reducible}
		If a path in $C(S)$ is dot reducible, then it's reducible.
	\end{lemma}
	
	\begin{proof}
	Since $\gamma$ is disjoint from $\alpha_0$ and $\alpha_n$ and $\alpha_0$ and $\alpha_n$ fills $S$, $\gamma$ has to lie inside one of the disks bounded by $\alpha_0\cup\alpha_n$. Let $U$ be a neighbourhood of $\gamma$ and $\alpha_i$ be one of the curves that has less intersections with $\gamma$ after the surgery, and $\alpha_{i_1},...\alpha_{i_m}$ are the arcs we get when cutting along $U$. Due to the fact that new $\alpha_i$ intersects less, at least one of the arcs are dropped. Because $\alpha_i$ is in minimal position, all the arcs have to intersect with $\alpha_0\cup\alpha_n$, and the new arcs added in $U$ does not intersect with $\alpha_0$ or $\alpha_n$, the complexity is reduced.
	\end{proof}
	
	However being reducible does not necessarily means dot reducible, since even if the complexity can be reduced, we may not find a surgery to do that.
	
	A \emph{maximum dot irreducible dot graph} is a dot graph such that the corresponding path in the curve complex is dot irreducible but if we add any point(s) to the dot graph, it becomes dot reducible.  Notice the number of dots in maximum dot irreducible dot graph may not be the most among all dot irreducible dot graphs of curves with the same length.
	
	We modify Proposition 3.1 in \cite{BMM}:
	
	\begin{lemma}
		Suppose $\sigma$ is a sequence of elements of the set $\{1,\cdots,n-1\}$. If there exists $k$ such that $\sigma$ has more than $n(n-1)$ entries equal to $k$, then $\sigma$ is dot reducible.
	\end{lemma}
	\begin{proof}
	Suppose $\sigma$ is already in sawtooth form. Let $\bigg[\sla{p_1}{q_1}, \cdots, \sla{p_k}{q_k}\bigg]$ denote the ascending segments of $G(\sigma)$ that has element $i$, ordered from left to right.
		
	If $p_i\not=p_j$ or $q_i\not=q_j$ for all $i,j$, we have $k\le n(n-1)$ due to the pigeon-hole rule. Then it's sufficient to show if there exists $p_i\not=p_j$and $q_i\not=q_j$, it's dot reducible.
	
	There will be half boxes that lies between $\sla{p_i}{q_i}$ and $\sla{p_j}{q_j}$. If any of them is a real box, then it's already dot reducible. Assume all half boxes are not real box, then for left half boxes the minimal of right segment will be lower than the left segment; and for right half boxes the maximum of the right segment will be higher than the left segment. Since for the start and the end the maximum and the minimal keep the same, there will be both left half boxes and right half boxes. Further there will be a left half box that is adjacent to a right half box, and they all have the elemental $k$, then we can perform the double hexagon surgery.
	\end{proof}
	From that lemma, we know maximum dot irreducible dot graph is well-defined.

	First we discuss some properties of them:
	
	\begin{lemma}
		There are finite many spindles for given $n$, and the spindles with the most elements are upper and lower complete triangles with $\frac{n(n-1)}{2}$ elements.
	\end{lemma}
	
	\begin{lemma}
		Spindles are dot irreducible.
	\end{lemma}
	
	We first prove lower complete triangles are dot irreducible (similarly upper complete triangles are also dot irreducible). To prove this we use induction. when $n=3$ we have the dot graph $[\sla{1}{2},\sla{1}{1}]$ which is dot irreducible. Suppose it is right for $n=k$ then for $n=k+1$, if we have a surgery with number $p$, when $p>1$ then we just consider the small triangle with distance $n+1-p<k$ and apply induction hypotheses. When $p=1$ there has to be a number 2 lies inside both 1 so it also has to be modified, then it goes to the case $p=2>1$.
	
	For spindles, using the argument above we know there can't be surgeries between two points that are both in the upper or the lower triangle. For one segment in the upper triangle and one in the lower triangle, there also can't be surgeries since they does not share a point at all.
	
	Now we can state the main result of the paper:
	
	\begin{thm}
		Maximum dot irreducible dot graph $G(\sigma)$ in sawtooth form is a (possible degenerated) spindle.
	\end{thm}
	\begin{proof}
	We prove the theorem in the following steps:
	
	Step 1: Show that there are numbers $1,2,...,n-1$ in $G(\sigma)$. If not, we can add the missing numbers at the leftmost of the got graph and it's still dot irreducible, contradicting with maximal.
	
	Step 2: Show that there's no horizontal line (i.e.  $\bigg[\cdots,\sla{p_i}{q_i}, \sla{p_{i+1}}{q_{i+1}},\cdots\bigg]$ where $q_i=p_{i+1}$). If there's horizontal lines, we can do a surgery on it immediately.
	
	Step 3: Show that the longest ascending segment is exactly $\sla{1}{n-1}$. Let $e$ be the longest ascending segment that's end at $n-1$. Let $e$ be $\sla{p}{n-1}$. It's sufficient to show $p=1$. If $p>1$, we extend the segment to $\sla{1}{n-1}$. If the new dot graph is not dot reducible, then we have a surgery that involves the added points. Assume $q<p$ is the point that's connected with another $q$ at segment $e'$.f
	
	If $e'$ is left to $e$, then we know there is element $q$ that in the left of $e$. If there's no $p$ in the left of $e$, then we can combine $e$ and the rightmost segment that is in the left of $e$ to have a new ascending segment. If there's $p$, we find the rightmost $p$ that's in the left of $e$, then we have a box surgery with the original $e$, contradict. 
	
	If $e'$ is right to $e$, consider the point $q+1$ on the new $e$, since it's blocking the surgery connecting $q$ in $e$ and $e'$, it has to be modified in the surgery. If it's connect to a point left to $e$, then we go back to the first case. Assume $q+1$ in the new $e$ is connected to some point in the right, using induction we know $n-1$ is connected to some point in the right. Let the point is in the segment $e''$ ($\sla{p''}{n-1}$). Because in our assumption $e$ ($\sla{p}{n-1}$) is the longest we have $p''\ge p$ so we have a box with segments original $e$ and $e''$. Using Lemma 3.6 in \cite{BMM} we can find an innermost box or hexagon to perform a surgery with original $e$, contradict.
	
	Notice the argument will also be valid if we discuss longest ascending segment that's start at 1, we just need to argue from $q$ to 1 in the induction part this time.
	
	Step 4: Show that if we have $e$ ($\sla{1}{k}$) and $e'$ ($\sla{1}{k-1}$) right to $e$ in $G(\sigma)$, then there's no segments $\sla{p}{q}$ ($q\le k$) lie between them. If $q=k$ then we find the leftmost such segment then there's a box between that segment and $e$ and we apply Lemma 3.6 in \cite{BMM}. Similarly, if $p=1$ then we find the rightmost such segment then there's a box between that segment and $e'$. Assume for segments lies between we have $p>1$ and $q<k$ then we pick the leftmost segment $\sla{p_1}{q_1}$ and rightmost segment $\sla{p_2}{q_2}$ (can be degenerated). Then we can apply the double hexagon surgery.
	
	Notice we can similarly have if we have $e$ ($\sla{k}{n-1}$) and $e'$ ($\sla{k-1}{n-1}$) right to $e$ in $G(\sigma)$, then there's also no segments between them.
	
	Step 5: Show that if there's no point or only segments that ends at some point larger than $k$ in the left of $\sla{1}{k}$ segment in $G(\sigma)$, then the segment just right to it has to be $\sla{1}{k-1}$. The proof here is very similar to what we did in Step 3. We find the longest ascending segment starting at 1 and show it's $\sla{1}{k-1}$. If it ends larger than $k-1$ then we will have a box with the segment $\sla{1}{k}$. If it ends smaller than $k-1$ then we can extend it to $\sla{1}{k-1}$ and apply the argument in Step 3. We just need to know there can't be any surgery between the segments in the right of $\sla{1}{k}$ and the segment $\sla{1}{k}$ or segments in the left because a point larger or equal to $k$ will block.
	
	Now we have an segment $\sla{1}{k-1}$ in the right. If it's not just right to $\sla{1}{k}$, then it will contradict with Step.4.
	
	Step 6: Show that if there's no point in the left of the segment $\sla{1}{n-1}$ we find in Step 3, then $G(\sigma)$ is a lower complete triangle. To prove this, we just need to use Step 5 repeatedly. Similarly, if there's no point in the right of the segment $\sla{1}{n-1}$, we have a higher  complete triangle.
	
	Step 7: Show that $G(\sigma)$ is a spindle. Consider all segments that's right to the segment $\sla{1}{n-1}$. Suppose the highest number is $p$, then we find all the segments ends at $p$ and apply the argument in Step 3. Then we have the segment is exactly $\sla{1}{p}$ and using similar argument as Step 4 we have the segment $\sla{1}{p}$ is just right to $\sla{1}{n-1}$. Similarly we have an segment $\sla{q}{n-1}$ just left to $\sla{1}{n-1}$. If $p\ge q-1$ we can have a double box surgery and if $p<q-2$ we can add the segment $\sla{1}{p+1}$ between $\sla{1}{n-1}$ and $\sla{1}{p}$ and it's still dot irreducible. After all we have $p=q-2$ so it's a spindle.
	\end{proof}
	
	Instantly we have the following corollaries
	
	\begin{cor}
		The maximal total intersecting number between a dot irreducible path and the reference arc $\gamma$ is $\frac{n(n-1)}{2}$ where $G(\sigma)$ is a lower (or higher) complete triangle.
	\end{cor}
	
	It is mentioned \cite{Jin} that the existence of super efficient geodesics will be justified by the geodesics of minimal total complexity. If a geodesic is dot irreducible with all reference arcs then it's irreducible and of minimal complexity. Theorem 1.1 of \cite{Jin} states for super efficient geodesics, we have $|\alpha_1\cup\gamma|\le 44$ for $g=2$ and $|\alpha_1\cup\gamma|\le 15(6g-8)$ for $g\ge 3$. For two consecutive intersections of $\alpha_1$ and $\gamma$. It's also super efficient so we can use the same argument and have $|\alpha_2\cup\gamma|\le 44^2$ or $(15(6g-8))^2$. By induction we have $|\alpha_i\cup\gamma|\le 44^i$ or $(15(6g-8))^i$ (by symmetry, it's also true for $|\alpha_{n+1-i}\cup\gamma|$). So we can define \emph{exponent} shape in the dot graph:
	
	\begin{de}
	    \emph{A} lower exponent shape of base $k$ \emph{in the dot graph is a shape satisfying the following:
	    \begin{enumerate}
		\item There are at most $k$ entries in the dot graph that equals $1$.
		\item Between two entries of $i$, there are at most $k$ entries that equals to $i+1$. This also holds before the first $i$ and after the last $i$.
		\end{enumerate}
		We can define an} upper exponent shape \emph{similarly, just replace $1$ with $n$ and $i+1$ with $i-1$.}
	\end{de}
	
	If $g=2$, we can take $k=44$ and if $g>2$ we can take $k=15(6g-8)$. So now we can improve our corollary for minimal complexity paths:
	
	\begin{cor}
		The maximal total intersecting number between a minimal complexity path and the reference arc $\gamma$ is $\sum_{i=1}^n \min(44^i,n-i)$ for $g=2$ and $\sum_{i=1}^n \min((15(6g-8))^i,n-i)$ for $g>2$ .
	\end{cor}
	
	Taking the shape into consideration, we will have
	\begin{de}
	    \emph{A} comb \emph{shape is a dot graph which consists of a spindle in the middle and upper exponent shape in the top and a lower exponent shape in the bottom. (See Fig. \ref{figure: comb})}
	\end{de}
	
	And we immediately have a corollary:
	\begin{cor}
	    Maximum dot irreducible dot graphs of geodesics with minimal complexity are combs.
	\end{cor}
	
	    \begin{figure}[h]
\labellist

\pinlabel ${\rm lower \ exponent}$ at 1265 70

\pinlabel ${\rm upper \ exponent}$ at 300 850

\pinlabel ${\rm spindle}$ at 750 550

\endlabellist
    \centering
    \includegraphics[width=1.05\linewidth]{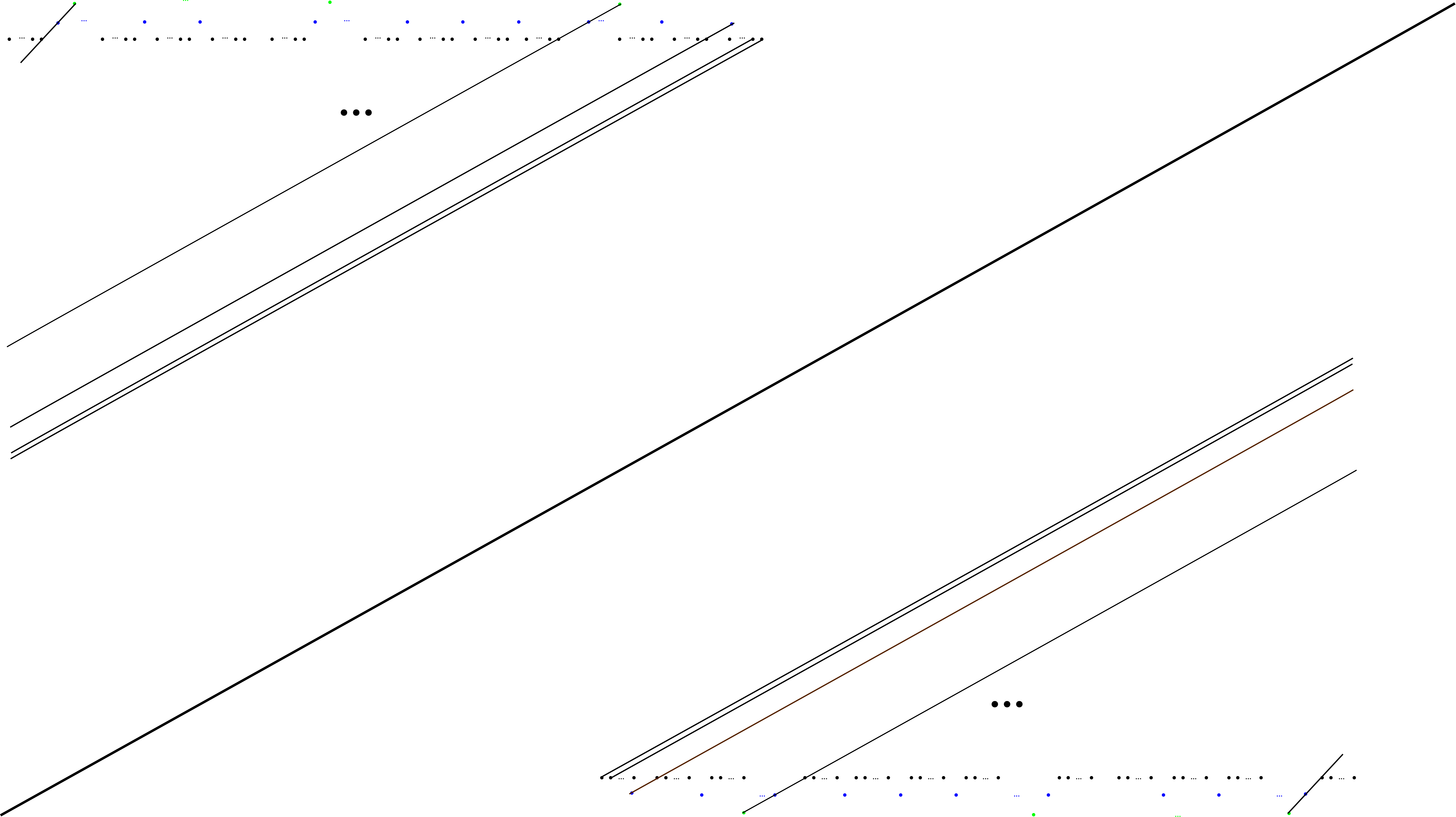}
		\caption{ A comb that is made up of three parts: the upper exponent shape, the spindle and the lower exponent shape}
		\label{figure: comb}
	\end{figure}
	
	For general dot irreducible dot graphs we have:
	
	\begin{prop}
		Every dot irreducible dot graph $G(\sigma)$ is a subgraph of a maximal dot irreducible graph.
	\end{prop}
		
	Notice the converse may not be right. $[\sla{1}{1},\sla{1}{1}]$ is dot reducible however when we add one point $[\sla{1}{2},\sla{1}{1}]$ is dot irreducible.
	
	Since for all spindles, there exists a unique entry $i$ such that there's only one $i$ in the dot graph. So we have the corollary immediately:
	
	\begin{cor}
	    There exists a curve in any dot irreducible geodesics with respect of the reference arc such that it intersects the reference arc once and curves with distance $d$ with that curve intersect the reference arc at most $d+1$ times.
	\end{cor}
	
	Consider the vertex in the dot irreducible geodesics in the last corollary and the ascending segment passing through that vertex, we can divide the dot graph into three parts: a subgraph of an upper triangle, an ascending segment passing through that vertex, and a subgraph of an lower triangle (See Fig.1). Suppose the lower triangle part is $\bigg[\sla{p_1}{q_1}, \cdots, \sla{p_k}{q_k}\bigg]$. If $q_i\ge q_{i+1}$ then the dot graph will be reducible using Step.4 in the proof of Theorem 7. So the sequence $q_1,\cdots,q_n$ must be strictly  decreasing. With that fact we can define \emph{slope} of the lower triangle part as $\frac{q_n-a_1}{n-1}$ and the value will be no larger than $-1$. Similarly if the upper triagnle part is $\bigg[\sla{s_1}{t_1}, \cdots, \sla{s_k}{t_k}\bigg]$ the sequence $s_1,\cdots,s_n$ is strictly  decreasing, so we similarly define the slope $\frac{s_n-s_1}{n-1}$ whose value is also always no larger than $-1$.

	\section{Finding possible surgeries for dot graphs}
	\label{section: finding surgeries}
	From the discussion above we know if a (sawtooth) dot graph is not a subgraph of any maximal dot irreducible graphs, then it will be reducible. However since the inverse does not hold, we need to invent an algorithm to determine whether a graph is reducible.
	
	Let $p_i$ and $p_{i+1}$ be two adjacent points $p$ in the (sawtooth) dot graph. If there's no $p+1$ or $p-1$ lies in between, then we can do a surgery with this two point (actually this is a horizontal line in the dot graph). If there's $p+1$ or $p-1$, then we say it's blocking the surgery.
	
	For surgeries we have the following observations:
	
	(1) In one surgery only one pair of $p$ can be connected. That's because we don't know which points $p$ are left after connecting one pair of $p$, so other surgeries may not be possible.
	
	(2) For a group of surgery steps, if for every step that connect $p_i$ and $p_{i+1}$ there's at least one $p+1$ or $p-1$ in the other steps. We call the group \emph{connected} surgery steps. Notice one group of steps can only remove one point like the hexagon surgery.
	
	So we have the following ways to deal with blocking at $p+1$:
	
	(1) If there's only one $p+1$ in between, we can either connect the point to another point or remove that point. For both of them we need to check $p+2$ and $p$ and notice the observations we mentioned. If there's only one $p+1$ in the dot graph, then the surgery is impossible.
	
	(2) If there are two $p+1$ in between, we need to connect one of them to another point and remove the other one, then we also need to check $p+2$ and $p$ and notice the observations we mentioned.
	
	(3) If there are three or more $p+1$ in between, the surgery is impossible.
	
	For $p-1$ we use the similar way to check until we get rid of all blocking or find the surgery impossible.
	
	\begin{figure}[h]
		\scalebox{.7}{\includegraphics[angle=0,origin=c]{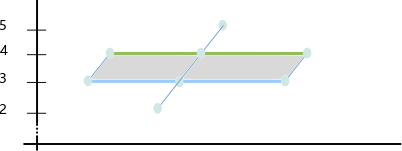}}
		\vspace{0cm}
		\caption{Pierced box $\bigg[\sla{3}{4},\sla{2}{5},\sla{3}{4}\bigg]$ and its possible surgeries}
		\label{surgery example}
	\end{figure}
	
	For example, for pierced box $\bigg[\sla{3}{4},\sla{2}{5},\sla{3}{4}\bigg]$ we can search for surgeries with point 3 with this algorithm:
	
	There are three $3'{\rm s}$ in this dot graph. If we combine the leftmost two $3'{\rm s}$, a $2$ and a $4$ will be blocking. For the $4$, we can either combine it with the $4$ in the middle, or we can combine the two rightmost $4'{\rm s}$ and drop the left one. However since the $2$ blocking is the unique $2$ in the dot graph, the surgery is impossible.
	
	Meanwhile, if we combine the rightmost two $3'{\rm s}$, a $4$ is blocking. We can combine that $4$ to either the left or the right one. If we connect to the left one then the only blocking $3$ is already combined. So the surgery is possible and the dot graph is dot reducible.  If we connect to the right one then there's a unique $6$ blocking making the surgery impossible.
	
	\section{Reference tree}
	\label{section: tree}
	
	Sometimes, it is impossible to find a single reference arc in one disk bounded by $\alpha_0\cup\alpha_n$ such that it intersect with all segments of $\alpha_1,...,\alpha_n$ inside that disk. (notice as a reference arc, it has to be in minimal position with all $\alpha_1,...,\alpha_n$) Below is an example:
	\begin{figure}[h]
		\scalebox{.4}{\includegraphics[angle=0,origin=c]{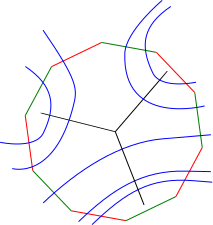}}
		\vspace{0cm}
		\caption{Red are arcs of $\alpha_0$, green are arcs of $\alpha_n$ while blue are segments of $\alpha_1,...,\alpha_n$}
		\label{reference tree}
	\end{figure}
	
	We can't find a curve to connect all the segments without nugatory crossings, however we can find a set of reference arcs (blue). Any of the three arcs or two arcs combined together is a reference arc of the green segments. We call the minimal collection of such reference arcs a \emph{reference graph}.
	
	\begin{prop}
		Any reference graph has no loops.
	\end{prop}
    \begin{proof}
	Suppose there's a loop, then it will bound a small polygon inside the disk bounded by $\alpha_0\cup\alpha_n$. Let $\alpha_i$ be a curve that intersect with a side of that polygon (such curve must exist, otherwise the graph is not minimal), then it must have another intersection. If it is on the same side, then the crossing is nugatory and can be reduced; otherwise we can remove a neighbourhood of that intersection from the graph and we have a smaller  reference graph, contradiction.
    \end{proof}
	With that result, we can call a reference graph \emph{reference tree}.  A reference tree is \emph{dot irreducible} if any reference arc from combining two arcs in the tree is dot irreducible. Here's an example of dot irreducible reference tree:
	
	\begin{figure}
		\scalebox{.4}{\includegraphics[angle=0,origin=c]{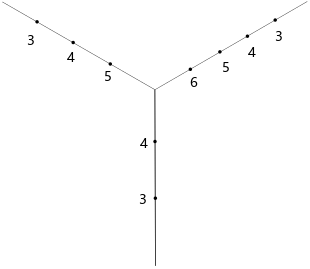}}
		\vspace{0cm}
		\caption{Dot irreducible reference tree}
		\label{dot irr tree}
	\end{figure}

	Three arcs from combining arcs are $\bigg[\sla{3}{5}, \sla{4}{4}, \sla{3}{3},\bigg]$, $\bigg[\sla{3}{6}, \sla{4}{4}, \sla{3}{3},\bigg]$ and $\bigg[\sla{3}{6}, \sla{5}{5},\sla{4}{4}, \sla{3}{3},\bigg]$, all of them are dot irreducible.
	
\end{document}